\documentclass[11pt,letterpaper]{amsart}
\pdfoutput=1
\usepackage{geometry}                
\usepackage{graphicx}
\usepackage{amsfonts,amsthm,amsmath,amssymb,latexsym, amscd, euscript}
\usepackage{graphicx,color}
\usepackage[all]{xy}
\usepackage{epsfig}
\usepackage{labelfig}
\usepackage{mathrsfs}

\usepackage{amssymb}
\usepackage{epstopdf}
\theoremstyle{plain}
\newtheorem{thm}{Theorem}[section]

\newtheorem{prop}[thm]{Proposition}
\newtheorem{lem}[thm]{Lemma}

\makeatletter
\@namedef{subjclassname@2020}{\textup{2020} Mathematics Subject Classification}
\makeatother

\title[Teichm\"uller extremal maps]{Teichm\"uller extremal maps on infinite Riemann surfaces}

\thanks{The author was partially supported by National Science Foundation award DMS-2521870, Simons Collaboration Grant award 00012869, and 
 PSC-CUNY grant 67366-00 55. The author is grateful to the Institute for Advanced Study and AMIAS
Member Fund for their support.}

\author{Dragomir \v Sari\' c}

\address[Dragomir \v Sari\' c] {Current address: Institute for Advanced Study, 1 Einstein Drive, Princeton, NJ 08540, USA}
	\email{dsaric@ias.edu}
\address{PhD Program in Mathematics, The Graduate Center, CUNY \\ 365 Fifth Ave., N.Y., N.Y., 10016 and \newline 
Department of Mathematics, Queens College, CUNY\\ 65--30 Kissena Blvd., Flushing, NY 11367, USA.}

\email{Dragomir.Saric@qc.cuny.edu}

\begin{document}


\begin{abstract} Let $X=\mathbb{D}/\Gamma$ be a Riemann surface with $\Gamma$ of the first kind. We establish a necessary and sufficient criterion for $[f]\in T(X)$ to have a Teichm\"uller-type extremal map. 
 \end{abstract}

\maketitle

\section{Introduction}

We consider a Riemann surface $X=\mathbb{D}/\Gamma$, where $\Gamma$ is a Fuchsian group of the first kind (i.e., the limit set of $\Gamma$ is the unit circle $S^1$). The Teichm\" uller space $T(X)$ of $X$ consists of equivalence classes $[f]$ of quasiconformal maps $f:X\to Y$, where two quasiconformal maps are equivalent if one is homotopic to a post-composition of the other by a conformal map  (see \cite{Ahlfors, GardinerLakic1, Lehto}). The homotopy class $[id]$ of the identity map $id:X\to X$ is called the {\it basepoint}.
Our considerations are valid for infinite Riemann surfaces (see \cite{Bas, BS}). 

The Teichm\"uller distance between $[id]$ and $[f]$ in $T(X)$ is one-half of the logarithm of the minimal quasiconformal constant of the maps in the homotopy class $[f]$. A map with the minimal quasiconformal constant in its homotopy classes is called {\it extremal}. When $X$ is a compact Riemann surface, each homotopy class has a unique extremal map which is given by the horizontal stretching in the natural parameter of a (finite-area) holomorphic quadratic differential on $X$. Such extremal maps are said to be of Teichm\"uller-type. When $X$ is neither a compact surface nor a compact surface minus finitely many points, there are homotopy classes of quasiconformal maps that have extremal maps not of Teichm\"uller-type (see \cite{BLMM, Lehto, ReichStrebel, reichstrebel1} and reference therein). Some of these classes may have even more than one extremal map (for example, see \cite{BLMM}). However, the set of points $[f]\in T(X)$ which have Teichm\"uller-type extremal maps is open and dense (see \cite{GardinerLakic1}). This was proved using the Strebel Frame Mapping Condition (see \cite{GardinerLakic1}), which gives a sufficient condition for a class $[f]$ to contain a Teichm\"uller-type map.

We give a complete characterization for $[f]\in T(X)$ to contain a Teichm\"uller-type extremal map (which is necessarily uniquely extremal). Let $A(X)$ be the space of all finite-area holomorphic quadratic differentials on $X$. For a non-trivial (i.e., not constantly zero) $\varphi\in A(X)$, denote by $\|\varphi\|$ the $L^1$-norm and by $h_{\varphi}$ the measured horizontal foliation of $\varphi$. 
Given $[f]\in T(X)$, we consider the mapping by heights
$$
f_{\#}:A(X)\to A(Y),
$$
where $f:X\to Y$. Let $f_*(h_{\varphi})$ be the push-forward of $h_{\varphi}$ to the image surface $Y$. Then $f_{\#}(\varphi )$ is the unique finite-area holomorphic quadratic differential $\psi\in A(Y)$ whose horizontal foliation is homotopic to $f_*(h_{\varphi})$. The existence of $f_{\#}(\varphi )=\psi $ is established in \cite[Theorem 1.6]{Saric-realization} and \cite[Theorem 1.1]{Saric-ft}, and the uniqueness in \cite{Saric-heights}.

\begin{thm}
\label{thm:main}
Let $X=\mathbb{D}/\Gamma$ be a Riemann surface with $\Gamma$ of the first kind. Then $[f]\in T(X)$ admits a Teichm\"uller-type extremal map if and only if the supremum 
$$
\sup_{\varphi\in A(X)\setminus \{ 0\}}\max\Big{\{}\frac{\| f_{\#}(\varphi )\|}{\|\varphi\|}, \frac{\|\varphi\|}{\| f_{\#}(\varphi )\|}
\Big{\}}
$$ 
is achieved at some $\varphi_{\mathrm{max}}\in A(X)\setminus\{ 0\}$. 

When the supremum is achieved, the Teichm\"uller map stretches the horizontal direction in the natural parameter of $\varphi_{\mathrm{max}}$ by the amount equal to the supremum. 
\end{thm}

The above theorem is established for compact surfaces by Marden and Strebel \cite{MardenStrebel-Teich}, and they gave a new proof of the Teichm\"uller existence theorem for the compact surfaces using the fact that the above supremum is always achieved in finite-dimensional $A(X)$. For arbitrary Riemann surfaces, the supremum is achieved for maps that are homotopic to a Teichm\"uller-type map.

{\it Open problems:} Extend the above theorem to arbitrary Fuchsian group $\Gamma$. Characterize the Strebel/Busemann points of $T(X)$ in terms of the height map $f_{\#}:A(X)\to A(Y)$ (see Earle-Li \cite{EarleLi}). Characterize the homotopy classes that have uniquely extremal maps using the height map (see \cite{BLMM}). 

\section{A necessary condition for the existence of the Teichm\"uller extremal maps}

Let $f:X\to Y$ be a $K$-quasiconformal map, with $X$ and $Y$ two arbitrary Riemann surfaces that admit conformal hyperbolic metrics. In particular, we allow $X$ to be the unit disk $\mathbb{D}$ or any infinite Riemann surface. The results in this section do not require that the Fuchsian groups of $X$ and $Y$ are of the first kind. The map $f$ induces the height map $f_{\#}:A(X)\to A(Y)$ which is a bijection (see \cite[Theorem 4.5]{Saric-realization}, \cite[Theorem 1.1]{Saric-ft}). We note that $f_{\#}$ does not preserve the $L^1$-norm. We first establish that $f_{\#}$ maps the unit sphere of $A(X)$ between the two spheres in $A(Y)$ of radii $1/K$ and $K$.

\begin{lem}
\label{lem:quasi-inv}
Let $f:X\to Y$ be a $K$-quasiconformal map. Then, for all $\varphi\in A(X)$,  
$$
\frac{1}{K}\| f_{\#}(\varphi )\|_{L^1}\leq\| \varphi\|_{L^1}\leq K \| f_{\#}(\varphi )\|_{L^1},
$$
where $\|\cdot\|_{L^1}$ is the $L^1$-norm on the corresponding surface. 
\end{lem}

\begin{proof}
 By Gardiner and Lakic \cite{GardinerLakic}, given $f:X\to Y$, there exists a sequence of $K_n$-quasiconformal maps $f_n:X\to Y$ in the same Teichm\"uller class as $f$ that are $C^{\infty}$-maps with $\lim_{n\to\infty}K_n$ equal to the minimal dilatation of the Teichm\"uller class of $f$.
 
Let $h_{\varphi}$ denote the measured horizontal foliation of the holomorphic quadratic differential $\varphi$. Let $(f_n)_*(h_{\varphi})$ be the push-forward of $h_{\varphi}$ by the $C^\infty$-map $f_n$ to a foliation of $Y$. Note that the foliations $(f_n)_*(h_{\varphi})$, for all $n$, are equivalent (homotopic) because $f_n$ are in the same Teichm\"uller class. 

For a measured foliation $h$, we denote by $\mathcal{D}(h_{\varphi})$ its Dirichlet integral (see \cite{Saric-realization}).
Note that $\mathcal{D}((f_n)_*(h_{\varphi}))\leq K_n \mathcal{D}(h_{\varphi})=K_n \int_X|\varphi |<\infty$ (see Ahlfors \cite{Ahlfors}).
Denote by $f_{\#}(\varphi)=\psi\in A(Y)$ the unique holomorphic quadratic differential whose horizontal foliation is equivalent to the foliations $(f_n)_*(h_{\varphi})$ and the foliation $f_*(h_{\varphi})$ (which exists by \cite[Theorem 1.6]{Saric-realization}). Let $h_{\psi} $ be the horizontal foliation of $\psi$. 

Since the heights of $h_{\psi}$ and $(f_n)_*(h_{\varphi})$ are equal,  the Dirichlet's principle (see Strebel \cite[Theorem 24.5]{Strebel} and \cite[Theorem 3.2]{SaricShima}) gives
$$
\int_Y|\psi |=\mathcal{D}(h_{\psi}) \leq \mathcal{D}((f_n)_*(h_{\varphi})). 
$$
By the above two inequalities, we get 
$$
\int_Y|\psi |\leq (\lim_{n\to\infty}K_n) \int_{X}|\varphi |\leq K \int_{X}|\varphi |.
$$
The last inequality follows because $\lim_{n\to\infty} K_n$ equals the minimal quasiconformal constant of the Teichm\"uller class of $f$, which is less than or equal to $K$.

The opposite inequality is obtained by replacing $f$ with $f^{-1}$.
\end{proof}

The above lemma implies that $1/K\leq \frac{\| f_{\#}(\varphi )\|_{L^1}}{\|\varphi \|_{L^1}}\leq K$ for all $\varphi\in A(X)\setminus\{ 0\}$. We define
\begin{equation}
\label{eq:L}
L=\sup_{\varphi\in A(X)\setminus\{ 0\}}\{ \frac{\| f_{\#}(\varphi )\|_{L^1}}{\|\varphi \|_{L^1}},\frac{\|\varphi \|_{L^1}}{\| f_{\#}(\varphi )\|_{L^1}}\} .
\end{equation}
Note that $L$ does not have to be achieved for any $\varphi$ on the unit sphere of $A(X)$ because the unit sphere is not compact. This fact is a major difference between the Teichm\"uller spaces of infinite and finite area hyperbolic surfaces.

Assume that $f$ is a Teichm\"uller extremal map given by stretching in the natural parameter of $\varphi\in A(X)$. If $w=u+iv$ is a natural parameter coordinate (given by $z\mapsto \int_{z_0}^{z}\sqrt{\varphi (z)}dz$), then $f$ is given by $w=u+iv\mapsto Ku+iv$. The heights of $\varphi$ and $f_{\#}(\varphi )=\psi$ are equal and $\int_Y|\psi |=K\int_{X}|\varphi |$. Therefore we conclude

\begin{prop}
If $f:X\to Y$ is in the same class as an extremal Teichm\"uller map for $\varphi\in A(X)$, then the quantity $L$ is achieved at $\varphi$.
\end{prop}

\section{A sufficient condition for the existence of the Teichm\"uller extremal maps}

The main goal of this section is to prove that if $L$ is achieved on some $\varphi\in A(X)$, then $f$ is homotopic to a Teichm\"uller map. We assume that $X=\mathbb{D}/\Gamma$ with $\Gamma$ of the first kind. 

\subsection{Approximation by differentials with single cylinders}
Let $\varphi\in A(X)$ be a non-trivial finite area holomorphic quadratic differential on $X$. We construct an approximation of $\varphi$ by holomorphic quadratic differentials whose horizontal measured foliations consist of single cylinders, called the Jenkins-Strebel differentials. If $X\in O_G$, then every finite area holomorphic quadratic differential is approximated by a sequence of Jenkins-Strebel differentials on $X$ in $L^1$-norm (see \cite[Theorem 1.2]{Saric-realization}). However, not every $X=\mathbb{D}/\Gamma$ with $\Gamma$ of the first kind is in the class $O_G$. Therefore, we establish an approximation of $\varphi$ on the doubles of an increasing sequence of finite-type subsurfaces of $X$, and their images under the quasiconformal map $f:X\to Y$. Recall that a step curve for $\varphi$ is obtained by the concatenation of the horizontal and vertical arcs of $\varphi$. 

\begin{lem}
\label{lem:JS-app}
Let $f:X=\mathbb{D}/\Gamma\to Y$ be a quasiconformal map, with $\Gamma$ a Fuchsian group of the first kind. Fix a finite area holomorphic quadratic differential $\varphi$ on $X$ and let $\psi =f_{\#}(\varphi )\in A(Y)$. Let $\{ X_n\}$ be an exhaustion of $X$ by increasing finite-type subsurfaces whose boundary curves are step curves of $\varphi$ and let $\{ Y_n=f(X_n)\}$ be the corresponding exhaustion of $Y$. Denote by $\widehat{X}_n$ and $\widehat{Y}_n$ the doubled surfaces of $X_n$ and $Y_n$, respectively. 

Then $f:X_n\to Y_n$ extends to a quasiconformal map $f:\widehat{X}_n\to\widehat{Y}_n$ by the reflection in the boundaries of $X_n$ and $Y_n$. Moreover, there exist Jenkins-Strebel differentials $\varphi_n\in A(\widehat{X}_n)$ and $\psi_n=f_{\#}(\widehat{\varphi}_n)\in A(Y_n)$ invariant under the reflections in the boundaries of $X_n$ and $Y_n$ such that, as $n\to\infty$, 
$$
\int_{X_n}|\varphi -\varphi_n|\to 0
$$
and
$$
\int_{Y_n}|\psi -\psi_n|\to 0.
$$
\end{lem}

\begin{proof}
 Let $\mu\in ML_{\mathrm{int}}(X)$ be the geodesic lamination obtained by straightening the leaves of the horizontal measured foliation $h_{\varphi}$ of $\varphi$ (see \cite{Saric-realization}). 
As in \cite[\S 3]{Saric-realization}, consider an exhaustion of $X$ by finite-area geodesic subsurfaces $\{X_n\}$ such that the boundary curves of $X_n$ are 
 $\varphi$-step curves (i.e., curves made by concatenating horizontal and vertical $\varphi$-arcs). Consider the restriction of the horizontal foliation 
$h_{\varphi}$ of $\varphi\in A(X)$ to the subsurface $X_n$ and erase all leaves that can be homotoped to a single component of the boundary of $X_n$. Denote the partial measured foliation by $h_n$. Each leaf of $h_n$ corresponds to a unique leaf of $\mu\cap X_n$ (see \cite[\S 3 and Definition 3.4]{Saric-realization}).

The double Riemann surface $\widehat{X}_n$ over the boundary of $X_n$ is a finite type Riemann surface, i.e., a compact Riemann surface with finitely many points removed. The partial foliation $h_n\subset X_n$ together with its mirror image in $\widehat{X}_n\setminus X_n$ forms a partial foliation $\widehat{h}_n$ in $\widehat{X}_n$ that is proper (see \cite[Lemma 3.6]{Saric-realization}). 
By the realization theorem of Hubbard-Masur \cite{HubbardMasur}, there exists a unique holomorphic quadratic differential $\widehat{\varphi}_n$ on $\widehat{X}_n$ whose horizontal foliation realizes the partial foliation $\widehat{h}_n$ in the sense that they have equal heights on simple closed curves. 

Since $\widehat{h}_n$ is invariant under the mirror symmetry in $\widehat{X}_n$ across the boundary $\partial X_n$, it follows that $\widehat{\varphi}_n$ is also invariant under the mirror symmetry by the uniqueness of the realization. Then
\begin{equation}
\label{eq:double_half}
\int_{\widehat{X}_n}|\widehat{\varphi}_n|=2\int_{{X}_n}|\widehat{\varphi}_n|
\end{equation}
and
$$
\mathcal{D}_{\widehat{X}_n}(\widehat{h}_n)=2\mathcal{D}_{{X}_n}(\widehat{h}_n).
$$
Then by the Dirichlet principle \cite[Theorem 7.5]{Bourque} (see also \cite[Theorem 3.2]{SaricShima}), we have
$$
\int_{{X}_n}|\widehat{\varphi}_n|\leq \mathcal{D}_{{X}_n}(\widehat{h}_n)=\mathcal{D}_{{X}_n}({h}_n).
$$

Since the foliation $h_n$ is a subfoliation of the foliation $h_{\varphi}$, we have $\mathcal{D}_{{X}_n}({h}_n)\leq \mathcal{D}_{{X}}({h_{\varphi}})=\int_X |\varphi |$ and we conclude that
\begin{equation}
\label{eq:limsup-app}
\limsup_{n\to\infty} \int_{{X}_n}|\widehat{\varphi}_n|\leq \int_X|\varphi |.
\end{equation}
By \cite[Lemmas 3.7 and 3.8]{Saric-realization}, we have that $\widehat{\varphi}_n|_{X_n}$ converges uniformly on compact subsets of $X$ to $\varphi^*$. 

Define
\begin{equation}
\label{eq:app-hol-diff}
\xi_n(z)=\left\{
\begin{array}{rcl}
\widehat{\varphi}_n(z),&\ \mbox{for}&\ z\in X_n\\
0,&\ \mbox{for}&\ z\in X\setminus X_n .
\end{array}
\right .
\end{equation}
 By the above, $\xi_n(z)$ converges locally uniformly to $\varphi$. From (\ref{eq:limsup-app}) and  \cite[Lemma 3.9]{Saric-realization} we have, as $n\to\infty$,
 \begin{equation}
 \label{eq:L^1-conv-X_n}
 \int_X|\varphi -{\xi}_n|\to 0.
 \end{equation}

 Consider the holomorphic quadratic differential $\psi =f_{\#}(\varphi )$ induced by $f:X\to Y$. Denote by $\{ Y_n\}$ the exhaustion of $Y$ by finite area geodesic subsurfaces such that $Y_n=f(X_n)$. Let $e_n$ be the partial foliation of $Y_n$ obtained by restricting the horizontal foliation of $\psi$ to $Y_n$ and erasing leaves that can be homotoped relative endpoints to the boundary. Then we double $Y_n$ to a finite area surface $\widehat{Y}_n$ without boundary and $e_n$ to a partial foliation $\widehat{e}_n$. By the construction, the partial foliations $f_*\widehat{h}_n$ and $\widehat{e}_n$ are equivalent since both straighten to the same measured (geodesic) lamination on $\widehat{Y}_n$. Let $\widehat{\psi}_n$ be the integrable holomorphic quadratic differential on $\widehat{Y}_n$ whose horizontal foliation is equivalent to $\widehat{e}_n$ (see Hubbard-Masur \cite{HubbardMasur}). We note that the quasiconformal map $f:X\to Y$ restricts to a quasiconformal map $f:X_n\to Y_n$ that homeomorphically maps boundary curves to boundary curves. Then, the extension 
 $f:\widehat{X}_n\to\widehat{Y}_n$ is defined by the reflections in the boundary sides (namely, $f(r(z))=r[f(z)]$, where $r(z)$ is the reflection), which implies that the extended map is quasiconformal with the same quasiconformal constant as the original map $f:X_n\to Y_n$.
 
 Define
\begin{equation}
\label{eq:app-hol-diff}
\eta_n(z)=\left\{
\begin{array}{rcl}
\widehat{\psi}_n(z),&\ \mbox{for}&\ z\in Y_n\\
0,&\ \mbox{for}&\ z\in Y\setminus Y_n
\end{array}
\right .
\end{equation}
 and by the same reasoning as the above, we have, as $n\to\infty$,
 \begin{equation}
 \label{eq:L^1-conv}
 \int_Y|\psi -{\eta}_n|\to 0.
 \end{equation}

 The quadratic differentials $\widehat{\varphi}_n\in A(\widehat{X}_n)$ and $\widehat{\psi}_n\in A(\widehat{Y}_n)$ are finite-area. A differential on a finite surface is called {\it Jenkins-Strebel} if its regular horizontal trajectories are closed and homotopic to each other forming a single cylinder on the surface (see \cite{Strebel})  By \cite{Masur}, each differential can be approximated by Jenkins-Strebel holomorphic quadratic differentials in the $L^1$-norm (since $\widehat{X}_n$ and $\widehat{Y}_n$ are of finite area). We can choose the approximating differentials to correspond to each other under the homotopy class of the double of $f:X_n\to Y_n=f(X_n)$ and to be invariant under the reflection by the invariance of  $\widehat{\varphi}_n$ and $\widehat{\psi}_n$. More precisely, let $\widehat{\varphi}_{n,k}\in A(\widehat{X}_n)$ and $\widehat{\psi}_{n,k}\in A(\widehat{Y}_n)$ be the Jenkins-Strebel differentials with corresponding cylinders under $f$  such that
 $$
 \int_{\widehat{X}_n}|\widehat{\varphi}_{n}-\widehat{\varphi}_{n,k}|\to 0\ \ \mbox{and}\ \ 
 \int_{\widehat{Y}_n}|\widehat{\psi}_{n}-\widehat{\psi}_{n,k}|\to 0
 $$
 as $k\to\infty$ for each $n\in\mathbb{N}$.
 
Let $\xi_{n,k}$ be the quadratic differential which agrees with $\widehat{\varphi}_{n,k}$ on $X_n$ and is equal to zero on $X\setminus X_n$. Then there exists $k_n$ such that, as $n\to\infty$,
\begin{equation*}
\label{eq:app-nn}
\int_X |\varphi -\xi_{n,k_n}|\to 0.
\end{equation*}
Define the quadratic differential $\eta_{n,k}$ given by $\eta_{n,k}:=\widehat{\psi}_{n,k}$ on $Y_n$ and $\eta_{n,k}:=0$ on $Y\setminus Y_n$ we have, as $n\to\infty$,
$$
\int_Y |\psi -\eta_{n,k_n}|\to 0.
$$

The quadratic differentials $\varphi_n :=\widehat{\varphi}_{n,k_n}$ and $\psi_n:=\widehat{\psi}_{n,k_n}$ satisfy the statement of the lemma.
\end{proof}

\subsection{The heights of the negative of the maximum quadratic differential}
Let $\varphi_{\mathrm{max}}$ be the quadratic differential such that $L=\frac{\| f_{\#}(\varphi_{\mathrm{max}})\|_{L^1}}{\|\varphi_{\mathrm{max}}\|_{L^1}}$, where $L$ is the supremum in (\ref{eq:L}). Define $$\psi_{\mathrm{max}}:=f_{\#}(\varphi_{\mathrm{max}}).$$

In this subsection, we relate the heights of $(f_{\#})^{-1}(-\psi_{\mathrm{max}})$ to the heights of $-\varphi_{\mathrm{max}}$ which is a major step in the proof of Theorem \ref{thm:main}.

\begin{thm}
\label{thm:max-extremal}
Let $f:X\to Y$ be a quasiconformal map. If the supremum in 
\begin{equation}
\label{eq:L=sup}
L=\sup_{\varphi\in A(X)\setminus\{ 0\}}\{ \frac{\| f_{\#}(\varphi )\|_{L^1}}{\|\varphi \|_{L^1}},\frac{\|\varphi \|_{L^1}}{\| f_{\#}(\varphi )\|_{L^1}}\} 
\end{equation}
is achieved at some $\varphi_{\mathrm{max}}\in A(X)\setminus\{ 0\}$, then there exists $c>0$ such that
$$
f_{\#}(-\varphi_{\mathrm{max}} )=-cf_{\#}(\varphi_{\mathrm{max}} ).
$$
\end{thm}

\begin{proof} We prove the theorem under the assumption that $L=\frac{\| f_{\#}(\varphi_{\mathrm{max}})\|_{L^1}}{\|\varphi_{\mathrm{max}} \|_{L^1}}$. The proof for the other case is given by replacing $f$ with $f^{-1}$ and following the steps for the first case.

Define $$\psi_{\mathrm{max}} :=f_{\#}(\varphi_{\mathrm{max}} )$$ and $$\widetilde{\varphi_{\mathrm{max}}}:=(f_{\#})^{-1}(-\psi_{\mathrm{max}} ).$$ We need to prove that 
\begin{equation}
\label{eq:tilde=-var}
\widetilde{\varphi_{\mathrm{max}}}=-c\varphi_{\mathrm{max}} .
\end{equation} 

Recall that $X=\mathbb{D}/\Gamma$ with $\Gamma$ a Fuchsian group of the first kind.
 The proof of (\ref{eq:tilde=-var}) extends the idea in Marden and Strebel 
 \cite[Theorem 10.4]{MardenStrebel-Teich} from compact to infinite Riemann surfaces $X=\mathbb{D}/\Gamma$ with $\Gamma$ of the first kind using an approximation method developed in \cite[\S 3 and \S 4]{Saric-realization}. 
 
Let $\varphi_n\in A(\widehat{X}_n)$ and $\psi_n\in A(\widehat{Y}_n)$ be the Jenkins-Strebel differentials with the properties from Lemma \ref{lem:JS-app} with respect to $\varphi_{\mathrm{max}}$ and $\psi_{\mathrm{max}}$.
Let $R_n$ and $Q_n$ denote the horizontal cylinders of ${\varphi}_{n}$ and ${\psi}_{n}$, respectively. Denote the lengths of the closed horizontal trajectories in $R_n$ and $Q_n$ by $a_n$ and $a_n'$, and the heights by $b_n$ and $b_n'$, respectively. Since ${\varphi}_{n}$ and ${\psi}_{n}$ correspond to each other under the natural class of homeomorphisms obtained by doubling $f:X_n\to f(X_n)$, the heights of the cylinders $R_n$ and $Q_n$ are equal, i.e., $b_n=b_n'$. 

Let $L_n$ be defined by the equation $a_n'=L_na_n$ (see \cite{MardenStrebel}). Then we have
$$\int_{\widehat{X}_n}|{\varphi}_{n}|=a_nb_n$$ and $$\int_{\widehat{Y}_n}|{\psi}_{n}|=a_n'b_n'=L_na_nb_n.$$ 
Since ${\varphi}_{n}$ and ${\psi}_{n}$ are invariant under mirror symmetries, we have 
$\frac{1}{2}\int_{\widehat{X}_n}|{\varphi}_{n}| = \int_{{X}_n}|{\varphi}_{n}|$ and 
$\frac{1}{2}\int_{\widehat{Y}_n}|{\psi}_{n}| = \int_{{Y}_n}|{\psi}_{n}|$. Then by (\ref{eq:L^1-conv-X_n}) and (\ref{eq:L^1-conv}) we have
$$
\int_Y|\psi_{\mathrm{max}} |=\lim_{n\to\infty}\frac{1}{2} L_na_nb_n=(\lim_{n\to\infty} L_n)\int_X|\varphi_{\mathrm{max}} |.
$$
Since $\varphi_{\mathrm{max}}$ attains the maximum for $L$ in (\ref{eq:L=sup}) and $\psi =f_{\#}(\varphi_{\mathrm{max}} )$  we have that
$$
\lim_{n\to\infty}L_n=L.
$$

Let $z=x+iy$ be a local parameter on $X$ and let 
$$z_n=x_n+iy_n=\int_{z^*}^z\sqrt{{\varphi}_n(z)}dz$$ be the natural parameter of ${\varphi}_n$ on $\widehat{X}_n$. The local parameter $z$ is defined on $X_n$, and we extend it by reflection to $\widehat{X}_n\setminus X_n$ and keep the same notation for simplicity. 

Let $\widetilde{{\varphi}_{n}}=(f_{\#})^{-1}(-{\psi}_{n})$, where $f$ denotes the quasiconformal map obtained by doubling $f:X_n\to Y_n=f(X_n)$. Let
$$
\tilde{w}_n=\tilde{u}_n+i\tilde{v}_n:=\int_{z^*}^z\sqrt{\widetilde{{\varphi}_{n}}(z_n)}dz_n
$$
be the natural parameter on the cylinder $R_n$ for the holomorphic quadratic differential $\widetilde{{\varphi}_{n}}$. 

Let $\alpha_n$ be a closed horizontal trajectory of $R_n$ for the quadratic differential ${{\varphi}_{n}}$, and let $\alpha_n'$ be a closed horizontal trajectory of $Q_n$ for the quadratic differential ${{\psi}_{n}}$.
For the natural parameter $z_n=x_n+iy_n$ of ${\varphi}_{n}$, we have $dz_n=dx_n$ on $\alpha_n$. Then, for the natural parameter $\tilde{w}_n=\tilde{u}_n+i\tilde{v}_n$ of the quadratic differential 
$\widetilde{{\varphi}_{n}}(z)=(f_{\#})^{-1}(-{\psi}_{n})(z)$, we get
\begin{equation}
\label{eq:alpha_n_length}
\begin{split}
\int_{\alpha_n}|\widetilde{{\varphi}_{n}}(z_n)|^{\frac{1}{2}}dx_n=\int_{\alpha_n}
\sqrt{\Big{(}\frac{\partial \tilde{u}_n}{\partial x_n}\Big{)}^2 + \Big{(}\frac{\partial \tilde{v}_n}{\partial x_n}\Big{)}^2}dx_n \geq  \int_{\alpha_n}\Big{|}\frac{\partial \tilde{v}_n}{\partial x_n}\Big{|}dx_n\\ \geq \int_{\alpha_n}|d\tilde{v}_n|\geq h_{\widetilde{{\varphi}_{n}}}(\alpha_n) =h_{-{\psi}_{n}}(\alpha_n')=a_n'=L_na_n .
\end{split}
\end{equation}
Note that $h_{\widetilde{{\varphi}_{n}}}(\alpha_n) =h_{-{\psi}_{n}}(\alpha_n')$ follows by the definition of the height map. 
By integrating (\ref{eq:alpha_n_length}) with respect to $dy_n$ from $0$ to $b_n$, we obtain
\begin{equation}
\label{eq:R_n_area}
\begin{split}
\iint_{R_n} |\widetilde{{\varphi}_{n}}(z_n)|^{\frac{1}{2}}dx_ndy_n=\iint_{R_n} \sqrt{\Big{(}\frac{\partial \tilde{u}_n}{\partial x_n}\Big{)}^2 + \Big{(}\frac{\partial \tilde{v}_n}{\partial x_n}\Big{)}^2}dx_ndy_n \\ \geq  \iint_{R_n}\Big{|}\frac{\partial \tilde{v}_n}{\partial x_n}\Big{|}dx_ndy_n \geq L_na_nb_n=L_n\iint_{R_n}|{\varphi}_{n}|
\end{split}
\end{equation}

In order to take the limit as $n\to\infty$, we change the integration to a local parameter $z=x+iy$ on the surfaces $\widehat{X}_n$, which is obtained by the natural reflection of a local parameter on $X_n\subset X$. The change of variables gives 
\begin{equation*}
\begin{split}
|\widetilde{{\varphi}_{n}}(z_n)|^{1/2}=|\widetilde{{\varphi}_{n}}(z)|^{1/2} |dz/dz_n|, \\
 dx_ndy_n=|dz_n/dz|^2dxdy\ \mathrm{and}\\ |dz_n|=|{\varphi}_{n}(z)|^{1/2}|dz|.
 \end{split}
\end{equation*} 
From (\ref{eq:R_n_area}), we obtain
$$
\iint_{R_n} |
\widetilde{{\varphi}_{n}}(z)|^{1/2}|{\varphi}_{n}(z)|^{1/2}dxdy\geq L_n \iint_{R_n}|{\varphi}_{n}|
$$
and applying the Cauchy-Schwarz inequality gives
$$
\Big{(}\iint_{R_n} |
\widetilde{{\varphi}_{n}}(z)|dxdy\Big{)} \Big{(}\iint_{R_n} |
{{\varphi}_{n}}(z)|dxdy\Big{)}\geq L_n^2 \Big{(} \iint_{R_n}|{\varphi}_{n}(z)|dxdy\Big{)}^2.
$$
By (\ref{eq:double_half}), (\ref{eq:app-nn}) and $\lim_{n\to\infty}L_n=L$, after letting $n$ go to infinity, the above inequality gives
\begin{equation}
\label{eq:m1}
\begin{split}
\Big{(}\limsup_{n\to\infty}\iint_{R_n} |
\widetilde{{\varphi}_{n}}(z)|dxdy\Big{)}\geq 2L^2\iint_{X}|\varphi_{\mathrm{max}} (z)|dxdy
\\
=2L
\iint_Y|\psi_{\mathrm{max}} (w)|dudv.
\end{split}
\end{equation}
The second equality sign of (\ref{eq:m1}) follows by the assumption that $\varphi_{\mathrm{max}}$ obtains maximum, namely $$[\iint_Y|\psi_{\mathrm{max}} (w)|dudv]/[\iint_{X}|\varphi_{\mathrm{max}} (z)|dxdy]=L.$$

By Proposition \ref{prop:limsup_upper_bound} proved below, we have $$\limsup_{n\to\infty}\iint_{R_n} |
\widetilde{{\varphi}_{n}}(z)|dxdy\leq 2\iint_X|\widetilde{\varphi_{\mathrm{max}}}(z)|dxdy.$$ Together with (\ref{eq:m1}), we obtain
\begin{equation}
\label{eq:main=}
\iint_X|\widetilde{\varphi_{\mathrm{max}}}(z)|dxdy\geq L
\iint_Y|\psi_{\mathrm{max}} (w)|dudv.
\end{equation}
Since $f_{\#}(\widetilde{\varphi_{\mathrm{max}}})=-{\psi_{\mathrm{max}}}$ and $L=\max\{ \frac{\| f_{\#}(\varphi )\|_{L^1}}{\|\varphi \|_{L^1}},\frac{\|\varphi \|_{L^1}}{\| f_{\#}(\varphi )\|_{L^1}}\}$, the above inequality is equality and all the inequalities in the proof become equalities when $n\to\infty$. 

We claim that $\widetilde{\varphi_{\mathrm{max}}}=-L^2\varphi_{\mathrm{max}}$. The natural parameter $\tilde{w}_n$ of $\widetilde{{\varphi}_{n}}$ on $X$ converges to the natural parameter $\tilde{w}=\tilde{u}+i\tilde{v}$ of $\widetilde{\varphi_{\mathrm{max}}}$. To see this, recall that the heights of 
$\widetilde{{\varphi}_{n}}$ are the same as the heights of $-{\psi}_{n}$ at the curves corresponding under $f$. Since $-{\psi}_{n}$ converge locally uniformly to $-\psi_{\mathrm{max}}$, it follows that the heights of $-{\psi}_{n}$ converge to the heights of $\psi_{\mathrm{max}}$ (see \cite[page 162, Theorem 24.7 ]{Strebel}. By Proposition \ref{prop:limsup_upper_bound}, the $L^1$-norms of $\widetilde{{\varphi}_{n}}$ are bounded above. Therefore, a subsequence converges locally uniformly to a finite-area holomorphic quadratic differential whose heights on $X$ are the same as the heights of $\widetilde{\varphi_{\mathrm{max}}}$ (see \cite[Lemma 3.8]{Saric-realization}). By the injectivity of the mapping by heights (see \cite[Theorem 1.2]{Saric-heights}), it follows that the limit is $\widetilde{\varphi_{\mathrm{max}}}$. This implies the convergence of the natural parameters $\tilde{w}_n$ to $\tilde{w}$.

After $n\to\infty$ in (\ref{eq:R_n_area}), inequalities become equalities. We  obtain $\partial\tilde{u}/\partial {x^*}\equiv 0$, where ${z^*}={x^*}+i{y^*}$ is the natural parameter of $\varphi_{\mathrm{max}}$. This implies that the horizontal trajectories of $\widetilde{\varphi_{\mathrm{max}}}$ are orthogonal to the horizontal trajectories of $\varphi_{\mathrm{max}}$, which implies that one differential is a constant multiple of the other. Then, by (\ref{eq:L=sup}),
$$
\widetilde{\varphi_{\mathrm{max}}}=-L^2\varphi_{\mathrm{max}}
$$
and
$$
f_{\#}(-\varphi_{\mathrm{max}} )=-\frac{1}{L^2}\psi_{\mathrm{max}} .
$$
\end{proof}

\subsection{From maximum of the height function to the Teichm\"uller maps} 
We prove that the conclusion of Theorem \ref{thm:max-extremal} implies the desired existence of the Teichm\"uller map in the homotopy class of $f:X\to Y$. This establishes Theorem \ref{thm:main} from the Introduction.

\begin{thm}
\label{thm:Teich-map}
Let $f:X\to Y$ be a quasiconformal map. If there exists $\varphi_{0}\in A(X)\setminus\{ 0\}$ with
$$
f_{\#}(-\varphi_{0} )=-cf_{\#}(\varphi_{0} )
$$
for some $c>0$, 
then there is a Teichm\"uller extremal map homotopic to $f$ obtained by horizontal stretching in the 
natural parameter of $\varphi_{0}$.
\end{thm}

\begin{proof}
Let $\psi_{0}=f_{\#}(\varphi_{0} )\in A(Y)$. We lift the holomorphic quadratic differentials $\varphi_{0}\in A(X)$ and $\psi_{0}\in A(Y)$ to holomorphic quadratic differentials 
$\tilde{\varphi_{0}}$ and $\tilde{\psi_0}$ on the universal covers $\tilde{X}$ and $\tilde{Y}$. Fix conformal identifications of $\tilde{X}$ and $\tilde{Y}$ with the unit disk $\mathbb{D}$. Then $X=\mathbb{D}/\Gamma$ and $Y=\mathbb{D}/\Gamma_1$; the lift $\tilde{f}:\mathbb{D}\to\mathbb{D}$ of $f:X\to Y$ conjugates $\Gamma$ to $\Gamma_1$; and the holomorphic quadratic differentials $\tilde{\varphi_0}$ and $\tilde{\psi}_0$ are equivariant with respect to $\Gamma$ and $\Gamma_1$, respectively.

By \cite{MardenStrebel}, each regular horizontal trajectory of $\tilde{\varphi_0}$ and of $\tilde{\psi_0}$ has exactly two limit points on the unit circle $S^1$, and each limit point corresponds to one class of trajectory rays going to infinity for a fixed parametrization of the ray. The heights map $f_{\#}$ induces a bijective correspondence between  regular horizontal trajectories of $\varphi_0$ and $\psi_0$ (see \cite{Saric-heights}). There are countably many singular horizontal trajectories of $\tilde{\varphi_0}$ and of $\tilde{\psi_0}$ since they have countably many zeros in $\mathbb{D}$. 

Since $f_{\#}(-\varphi_0 )=\frac{1}{L^2}(-\psi_0 )$ the heights map $f_{\#}$ is also mapping the regular vertical trajectories of $\varphi_0$ onto the regular vertical trajectories of $\psi_0$. The lift $\tilde{f}$ induces a bijective heights map $\tilde{f}_{\#}$ between the sets of regular horizontal trajectories of $\tilde{\varphi_0}$ and $\tilde{\psi_0}$ as well as regular vertical trajectories of the two differentials. 

We construct a map $\tilde{g}:\mathbb{D}\to\mathbb{D}$ using the correspondence $\tilde{f}_{\#}$. By the uniqueness of the geodesics for the metric induced by a holomorphic quadratic differential in a simply connected domain (see \cite[page 72, Theorem 14.2.1]{Strebel}), each regular horizontal and vertical trajectory of a holomorphic quadratic differential $\tilde{\varphi_0}$ can intersect in at most one point in $\mathbb{D}$ (the same is true for $\tilde{\psi}$). Therefore, each point of the complement of the countably many singular horizontal and vertical trajectories is the intersection of a unique horizontal and a unique vertical regular trajectory of $\tilde{\varphi_0}$. We define $\tilde{g}$ to send this point to the intersection of the corresponding (under $\tilde{f}_{\#}$) horizontal and vertical trajectories of $\tilde{\psi_0}$. 

So far, the map $\tilde{g}$ is defined on the complement of the singular horizontal and vertical trajectories of $\tilde{\varphi_0}$. 
We establish that $\tilde{g}$ is continuous. Indeed, let $z_n\in\mathbb{D}$ be a sequence of points where $\tilde{g}$ is defined that converges to $z\in\mathbb{D}$. Let $\frak{h}_n$ be the horizontal trajectory of $\varphi_0$ that contains $z_n$ and let $\frak{v}_n$ be the vertical trajectory of $\varphi_0$ that contains $z_n$. Then $\frak{h}_n\cap\frak{v}_n=\{ z_n\}$. Let $\frak{h}$ be the limit horizontal trajectory of the sequence $\frak{h}_n$. The horizontal trajectory $\frak{h}$ is either regular or 
it limits to a zero of $\varphi_0$ in at least one end. In the latter case, we can take a subsequence of $\frak{h}_n$, if necessary, such that $\frak{h}_n$ are on the same side of $\frak{h}$. Then the limit of $\frak{h}_n$ consists of a concatenation of, possibly countably many, singular horizontal trajectories. In either case, the ends of the trajectory $\frak{h}$ consist of two distinct points. An analogous construction yields a sequence of regular vertical trajectories $\frak{v}_n$ that contain $z_n$ and converge on one side to a vertical trajectory $\frak{v}$ which contains $z$. The vertical trajectory $\frak{v}$ is either regular or it is a concatenation of, at most countably many, singular vertical trajectories, both ends of which converge to distinct points on $S^1$.

Then $\tilde{f}_{\#}(\frak{h}_n)$ and $\tilde{f}_{\#}(\frak{v}_n)$ are regular trajectories of $\tilde{\psi}$. The sequence $\tilde{f}_{\#}(\frak{h}_n)$ converges to a horizontal trajectory $\frak{h}^*$. The trajectory $\frak{h}^*$ is either regular, in which case $\tilde{f}_{\#}(\frak{h})= \frak{h}^*$, or it is a concatenation of singular horizontal trajectories of $\tilde{\psi_0}$. Analogous statements hold for $\tilde{f}_{\#}(\frak{v}_n)$ and its limiting vertical trajectory $\frak{v}^*$. Let $w_n=\tilde{f}_{\#}(\frak{h}_n)\cap \tilde{f}_{\#}(\frak{v}_n)$ and $w=\frak{h}^*\cap \frak{v}^*$. 

Assume that $z$ is not a zero of $\tilde{\varphi_0}$ and $w$ is not a zero of $\tilde{\psi_0}$. 
Let $R$ be an arbitrarily small rectangle in the natural parameter of $\tilde{\psi_0}$ whose center is $w$. By the convergence of 
$\tilde{f}_{\#}(\frak{h}_n)$ to $\frak{h}^*$ and $\tilde{f}_{\#}(\frak{v}_n)$ to $\frak{v}^*$, $\tilde{f}_{\#}(\frak{h}_n)\cap R$ and $\tilde{f}_{\#}(\frak{v}_n)\cap R$ are a horizontal and vertical segments when $n$ is large enough. Therefore $w_n\in R$ and $\lim_{n\to\infty}w_n=w$ which implies that $\tilde{g}$ extends to a continuous map from $\mathbb{D}$ minus the set of the zeros of $\tilde{\varphi_0}$ and $\tilde{\psi_0}$ onto itself. By the same reasoning, it follows that $\tilde{g}^{-1}$ is continuous and therefore $\tilde{g}$ is a homeomorphism outside a discrete subset of $\mathbb{D}$. 

In addition, the map $\tilde{g}$ is fixing the vertical direction and stretching the horizontal direction by a factor $L^2$ in the natural parameter of $\tilde{\varphi_0}$ because $f_{*}(
h_{-\varphi_0} )=\frac{1}{L^2}(h_{-\psi_0} )$. It follows that $\tilde{g}$ is an $L^2$-quasiconformal map, and it extends to the complementary discrete set because quasiconformal maps extend to isolated points. In fact, $\tilde{g}$ is the Teichm\"uller map for the differential $\tilde{\varphi_0}$ and it agrees with $\tilde{f}$ on $S^1$ by the construction. Since $\tilde{g}$ is invariant under $\Gamma$, then we obtain a Teichm\"uller map $g:X\to Y$ in the homotopy class of $f:X\to Y$. 
\end{proof}

\section{The continuity of the approximations with respect to the heights map}

Let $f:X\to Y$ be a quasiconformal map. Consider the mapping by heights $f_{\#}:A(X)\to A(Y)$ which assigns to each $\varphi\in A(X)$ a holomorphic quadratic differential $f_{\#}(\varphi )\in A(Y)$ such that $f_{*}(\nu_{\varphi})=\nu_{f_{\#}(\varphi )}$, where $\nu_{\varphi}$ is the measured geodesic foliation obtained by straigthening $h_{\varphi}$ and $\nu_{f_{\#}(\varphi )}$ is the measured geodesic lamination obtained by straigthening $h_{f_{\#}(\varphi )}$. We set $\psi =f_{\#}(\varphi )$.

The quasiconformal map $f:X\to Y$ extends by reflections to a quasiconformal map $f:\widehat{X}_n\to\widehat{Y}_n$. Denote by $X_n$ and $Y_n$ the halves of $\widehat{X}_n$ and $\widehat{Y}_n$ that lie in $X$ and $Y$, respectively. Let $g:=f^{-1}$ be the inverse quasiconformal map and let $\mu$ be its Beltrami coefficient.
Set
$$
g^t=g^{t\mu}:Y\to Y^t
$$
where $g^{t\mu}$ is the quasiconformal map whose Beltrami coefficient is $t\mu$ for $0\leq t\leq 1$. Then $g^1=f^{-1}$,  $g^0=id$ and $Y^t$ is the image Riemann surface $g^{t\mu}(Y)$. By definition, $g^0(Y)=Y$ and $g^1(Y)=X$. 

The double Riemann surfaces $\widehat{Y}_n^t$ are obtained by doubling the Riemann surfaces $g^{t}({Y}_n)=Y_n^t\subset Y^t$. The induced quasiconformal map from $\widehat{Y}_n$ to $\widehat{Y}_n^t$ will be denoted by $g^t$, for simplicity.

Let ${\varphi}_{n}\in A(\widehat{X}_n)$ and ${\psi}_{n}\in A(\widehat{Y}_n)$
be Jenkins-Strebel differentials with corresponding cylinders under $f$ from Lemma \ref{lem:JS-app} for $\varphi\in A(X)$ and $\psi =f_{\#}(\varphi )\in A(Y)$. 
To simplify the notation, we set
$$
q_n:={{\psi}_{n}}
$$
and
$$
p_n:={{\varphi}_{n}}.
$$
Note that $g^1_{\#}(-q_n)=p_n$ and define
$$
p_n^t=g^t_{\#}(-q_n).
$$

\begin{prop}
\label{prop:limsup_upper_bound}
Under the above notation, we have
$$
\limsup_{n\to\infty} \int_{\widehat{X}_n}|p_n|\leq 2\iint_X|{\varphi}|.
$$
\end{prop}

\begin{proof}
Assume, on the contrary, that 
$$
\limsup_{n\to\infty} \int_{\widehat{X}_n}|p_n|> 2\int_X|{\varphi}|.
$$
We seek a contradiction. Our method uses an idea from Lakic \cite[Lemma 3]{Lakic}. 

We note that $p_n^t$ have bounded $L^1$-norms on $\widehat{Y}_n^t$ because $g^t_{\#}(-q_n)=p_n^t$, $g_t$ is a quasiconformal map and the $L^1$-norms on $\widehat{Y}_n$ of $q_n$ are bounded. Since the heights of simple closed curves on $Y^t$ in the $p_n^t$-metric converge to the heights of the corresponding curves on $Y$ in the $(-\psi )$-metric, it follows that $p_n^t$ converge uniformly on compact subsets of $Y^t$ to an integrable holomorphic quadratic differential $p^t$ whose heights are equal to the heights of $\psi$ on $Y$. Indeed, the $L^1$-norms of the differentials $p_n^t$ are bounded by the Dirichlet principle and the fact that $g^{t\mu}$ have bounded quasiconformal constants. Then a subsequence of $p_n^t$ converges uniformly on compact subsets. By \cite[Theorem 24.7]{Strebel}, the heights of the subsequence converge to the heights of the limit quadratic differential. By the uniqueness of the heights function \cite{Saric-heights}, the limit differential $p^t\in A(Y^t)$ is independent of the subsequence. Therefore, the whole sequence converges to $p^t$.

Define
$$
A(t)=\limsup_{n\to\infty}\int_{\widehat{Y}_n^t}|p_n^t|-2\int_{Y^t}|p^t|
$$
and note that the following holds:
\begin{enumerate}
\item $A(t)$ is non-negative for all $0\leq t\leq 1$ and
$$
S=\sup_{t\in [0,1]}A(t)\leq 2\|\psi\|_{L^1(Y)}\frac{1+\|\mu\|_{\infty}}{1-\|\mu\|_{\infty}},
$$
\item $A(1)=\limsup_{n\to\infty} \int_{\widehat{X}_n}|p_n|-2\int_X|{\varphi}|>0$ and $A(0)=\limsup_{n\to\infty} \int_{\widehat{Y}_n}|q_n|-2\int_X|{\psi}|=0$, and
\item $\lim_{t\to 0}A(t)=0$.
\end{enumerate}
To see that $A(t)\geq 0$, note that by Fatou's lemma we have
$$
\limsup_{n\to\infty}\int_{\widehat{Y}_n^t}|p_n^t|\geq \liminf_{n\to\infty}\int_{\widehat{Y}_n^t}|p_n^t|\geq 2\int_{Y^t}|p^t|.
$$
The second part of (1) follows from 
\begin{equation*}
\begin{split}
\limsup_{n\to\infty}\int_{Y_n^t}|p_n^t|-\int_{Y^t}|p^t|\leq \limsup_{n\to\infty}\int_{Y_n^t}|p_n^t|\leq \limsup_{n\to\infty}K(g_t)\int_{Y_n}|q_n | \\
\leq \limsup_{n\to\infty}\int_{\widehat{Y}_n}|q_n |\frac{1+t\|\mu\|_{\infty}}{1-t\|\mu\|_{\infty}}\leq 2\|{\psi}\|_{L^1({Y})} \frac{1+t\|\mu\|_{\infty}}{1-t\|\mu\|_{\infty}}
\end{split}
\end{equation*}
The first condition in (2) is the initial assumption in this proof, and the second condition in (2) is by definition. To prove condition (3), note that, by the change of variables, 
\begin{equation}
\label{eq:A(t)to0}
\begin{split}
A(t)\leq \limsup_{n\to\infty} \|q_n\|_{L^1(\widehat{Y}_n)}\frac{1+t\|\mu\|_{\infty}}{1-t\|\mu\|_{\infty}} 
- 2\|q \|_{L^1(Y)}\frac{1-t\|\mu\|_{\infty}}{1+t\|\mu\|_{\infty}} \\
=2\| q\|_{L^1(Y)} \Big{(} \frac{1+t\|\mu\|_{\infty}}{1-t\|\mu\|_{\infty}}-\frac{1-t\|\mu\|_{\infty}}{1+t\|\mu\|_{\infty}}\Big{)}
\end{split}
\end{equation}
which implies the desired condition. 

The supremum $S$ is positive because $A(1)>0$. Since $\lim_{t\to 0}A(t)=0$ and $A(0)=0$, there exists $t_0\in (0,1]$ such that
\begin{equation}
\label{eq:almost_sup}
A(t_0)>S/2.
\end{equation} 

Define
\begin{equation}
h(t)=\iint_{Y^t}| p^t|
\end{equation}
and
\begin{equation}
h_n(t)=\iint_{Y^t_n}|p_n^t|.
\end{equation}

The functions $h(t)$ and $h_n(t)$ are $C^1$, and a variational formula for the Dirichlet integral from Lakic \cite[page 311]{Lakic} applies to arbitrary Riemann surfaces to give 
\begin{equation}
h'(t)=2\mathrm{Re}\iint_{Y^t}\frac{\mu}{1-|t\mu |^2} \frac{g^t_z}{\overline{g^t_z}}\circ (g^t)^{-1}  p^t
\end{equation}
and
\begin{equation}
h_n'(t)=2\mathrm{Re}\iint_{Y_n^t}\frac{\mu}{1-|t\mu |^2}\frac{g^t_z}{\overline{g^t_z}}\circ (g^t)^{-1}p_n^t.
\end{equation}

In order to have integration over the same space, we extend $p_n^t$ to be zero in $Y^t\setminus Y_n^t$ and keep the notation $p_n^t$. Whenever we have the integration of $p_n^t$ over $Y_n^t$, we formally replace it with the integration of $p_n^t$ over $Y^t$ without further mention.  

Then we have
\begin{equation}\begin{split}
A(t_0)=\limsup_{n\to\infty}\int_0^{t_0}[h_n'(t)-h'(t)]dt\\
\leq 2\limsup_{n\to\infty}\int_0^{t_0}\Big{|} \iint_{Y^t} \frac{\mu}{1-|t\mu |^2}\frac{g^t_z}{\overline{g^t_z}}\circ (g^t)^{-1}\Big{(}p_n^t-p^t\Big{)}\Big{|}dt\\
\leq \frac{2\|\mu\|_{\infty}}{1-\|\mu \|_{\infty}^2}\limsup_{n\to\infty} \int_0^{t_0} \iint_{Y_n^t} \Big{|}p_n^t-p^t\Big{|}dt
\end{split}
\end{equation}

Since $q_n$ approximate $q$ and by the change of the variables, there exists $M>0$ such that
$ \iint_{Y^t} {|}p_n^t-p^t{|}\leq M$. The Fatou's lemma gives
\begin{equation}
\label{eq:limsup-change}
\begin{split}
 \limsup_{n\to\infty}\int_0^{t_0}\iint_{Y^t} {|}p_n^t-p^t{|}dt
 =Mt_0-\liminf_{n\to\infty}  
 \int_0^{t_0} \Big{(} M-
 \iint_{Y^t} {|}p_n^t-p^t{|}\Big{)}dt
  \Big{)}\\
  \leq Mt_0-\int_0^{t_0} \liminf_{n\to\infty}\Big{(} M-\iint_{Y^t} {|}p_n^t-p^t{|}\Big{)}dt\\
  \leq 
  \int_0^{t_0}\limsup_{n\to\infty}\iint_{Y^t} {|}p_n^t-p^t{|}dt
\end{split}
\end{equation}

Then we obtain
\begin{equation}
\label{eq:A(t_0)}
\begin{split}
A(t_0)\leq \frac{2\|\mu\|_{\infty}}{1-\|\mu \|_{\infty}^2}\int_0^{t_0}\Big{(} \limsup_{n\to\infty}\iint_{Y^t} {|}p_n^t-p^t{|}\Big{)}dt\\
\leq \frac{2\|\mu\|_{\infty}}{1-\|\mu \|_{\infty}^2}\int_0^{t_0}\Big{(} \limsup_{n\to\infty}(\iint_{Y^t} {|}p_n^t-p^t{|}-|p_n^t|)\\ + \limsup_{n\to\infty}\iint_{Y^t}|p_n^t|\Big{)}dt.
\end{split}
\end{equation}

By the above inequality, since $p_n^t$ converges uniformly on compact subsets of $Y^t$ to $p^t$ and by an application of the Lebesgue Dominated Convergence Theorem to the first integral on the right, we get
\begin{equation}
\label{eq:A(t_0)-continuation}
\begin{split}
A(t_0)\leq 
\frac{2\|\mu\|_{\infty}}{1-\|\mu \|_{\infty}^2}\int_0^{t_0}\Big{(}\iint_{Y^t} -|p^t|
+ \limsup_{n\to\infty}\iint_{Y^t}|p^t|\Big{)}dt\\
=\frac{2\|\mu\|_{\infty}}{1-\|\mu \|_{\infty}^2}\int_0^{t_0}A(t)dt\leq \frac{2\|\mu\|_{\infty}}{1-\|\mu \|_{\infty}^2}t_0A(t_0)=
\frac{2\|\mu\|_{\infty}}{1-\|\mu \|_{\infty}^2}t_0S
\end{split}
\end{equation}

Since $A(t_0)>S/2$, the above inequality implies
$$
S/2< \frac{2\|\mu\|_{\infty}}{1-\|\mu \|_{\infty}^2}t_0S<\frac{2\|\mu\|_{\infty}}{1-\|\mu \|_{\infty}^2}S
$$
which gives 
$$
1-\|\mu\|_{\infty}^2<4\|\mu\|_{\infty}.
$$
The above inequality gives a contradiction when $\|\mu\|_{\infty}<1/5$.

Therefore the statement is true for all $f$ whose Beltrami coefficient has norm less than $1/5$. Since each quasiconformal map can be written as a composition of finitely many quasiconformal maps whose Beltrami coefficients have norms less than $1/5$ and the heights map of a composition is the composition of the height maps, the proposition follows.
\end{proof}


\begin{thebibliography}{Thua}

\vskip .5cm

\bibitem{Ahlfors}  Ahlfors, Lars V. {\it Lectures on quasiconformal mappings.} Second edition. With supplemental chapters by C. J. Earle, I. Kra, M. Shishikura and J. H. Hubbard. University Lecture Series, 38. American Mathematical Society, Providence, RI, 2006.

\bibitem{Bas} A. Basmajian, {\it  Hyperbolic structures for surfaces of infinite type}, Trans. Amer. Math. Soc. 336, no. 1, March 1993, 421-444.

\bibitem{BS} A. Basmajian and D. \v Sari\' c, {\it Geodesically complete hyperbolic structures}, Math. Proc. Cambridge Philos. Soc. 166 (2019), no. 2, 219-242.


\bibitem{Bers}  Bers, Lipman {\it Universal Teichm\"uller space.} Analytic methods in mathematical physics (Sympos., Indiana Univ., Bloomington, Ind., 1968), pp. 65-83, Gordon and Breach, New York-London-Paris, 1970.

\bibitem{Bourque}  Bourque, Maxime Fortier {\it The holomorphic couch theorem.} Invent. Math. 212 (2018), no. 2, 319-406.

\bibitem{BLMM} Bo\v zin, V.; Lakic, N.; Markovi\' c, V.; Mateljevi\' c, M. {\it 
Unique extremality.} 
J. Anal. Math. 75 (1998), 299-338. 

\bibitem{EarleLi}  Earle, Clifford J.; Li, Zhong {\it Isometrically embedded polydisks in infinite dimensional Teichm\"uller spaces.} J. Geom. Anal. 9 (1999), no. 1, 51-71. 

\bibitem{GardinerLakic} Gardiner, F. P.; Lakic, N. {\it Efficient smooth quasiconformal mappings.} Complex manifolds and hyperbolic geometry (Guanajuato, 2001), 197-205, Contemp. Math., 311, Amer. Math. Soc., Providence, RI, 2002.

\bibitem{GardinerLakic1}  F. Gardiner and N. Lakic, {\it Quasiconformal Teichm\" uller theory}, Mathematical Surveys and Monographs, 76. American Mathematical Society, Providence, RI, 2000. 




 \bibitem{HubbardMasur} J. Hubbard and H. Masur, {\it Quadratic differentials and foliations}, Acta Math. 142 (1979), no. 3-4,  221-274.
 
\bibitem{Jenkins} J. A. Jenkins, {\it On the existence of certain general extremal metrics}, Ann. of Math.,
66 (1957), 440-453.


\bibitem{Lakic} N. Lakic, {\it The minimal norm property for quadratic differentials in the disk}, Michigan Math. J. 44 (1997), no. 2, 299-316.

\bibitem{Lehto}  Lehto, Olli {\it Univalent functions and Teichm\" uller spaces.} Graduate Texts in Mathematics, 109. Springer-Verlag, New York, 1987. 


\bibitem{MardenStrebel} A. Marden and K. Strebel, {\it On the ends of trajectories},  Differential geometry and complex analysis, 195-204, Springer, Berlin, 1985.

\bibitem{MardenStrebel1} A. Marden and K. Strebel, {\it The heights theorem for quadratic differentials on Riemann surfaces}, Acta Math. 153 (1984), no. 3-4, 153-211.

\bibitem{MardenStrebel-Teich} A. Marden and K. Strebel, {\it A characterization of Teichm\"uller differentials}, J. Differential Geom. 37 (1993), no. 1, 1-29. 


\bibitem{Masur}  H. Masur, {\it The Jenkins-Strebel differentials with one cylinder are dense}, Comment. Math. Helv. 54 (1979), no. 2, 179-184.


\bibitem{ReichStrebel}  E. Reich and K. Strebel, {\it Extremal quasiconformal mappings with given boundary values}, Contributions to analysis (a collection of papers dedicated to Lipman Bers), pp. 375-391. Academic Press, New York, 1974.



\bibitem{reichstrebel1} E. Reich and K. Strebel, {\it Teichm\"uller mappings which keep the boundary point-wise fixed},  1971 Advances in the Theory of Riemann Surfaces (Proc. Conf., Stony Brook, N.Y., 1969) pp. 365-367 Princeton Univ. Press, Princeton, N.J.

\bibitem{Saric-heights} D. \v Sari\' c, {\it The heights theorem for infinite Riemann surfaces}, Geom. Dedicata 216 (2022), no. 3, Paper No. 33, 23 pp. 

\bibitem{Saric-realization} D. \v Sari\' c, {\it Quadratic differentials and foliations on infinite Riemann surfaces.} Duke Math. J. 173 (2024), no. 10, 1883-1930.

\bibitem{Saric-ft} D. \v Sari\' c, {\it Quadratic differentials and function theory on Riemann surfaces},  	 	arXiv:2407.16333 [math.DS].


\bibitem{SaricShima} D. \v Sari\' c and T. Shima, {\it Intersection numbers between horizontal foliations of quadratic differentials}, arXiv:2506.14943.


\bibitem{Strebel1} K. Strebel, {\it  \"Uber quadratische differentials mit geschlossen trajectorien and extremale
quasikonforme abbildungen}, In Festband zum 70, Geburstag von Rolf Nevanlinna,
Springer-Verlag, Berlin, 1966.

\bibitem{Strebel} K. Strebel, {\it Quadratic differentials}, Ergebnisse der Mathematik und ihrer Grenzgebiete (3) [Results in Mathematics and Related Areas (3)], 5. Springer-Verlag, Berlin, 1984. 


\end{thebibliography}
\end{document}